\documentclass[12pt, a4paper]{amsart}
\usepackage{color}
\usepackage{amscd,amssymb,graphics}
\usepackage[hidelinks]{hyperref}
\usepackage{amsfonts}
\usepackage{amsmath}
\usepackage{amsxtra}
\usepackage{latexsym}
\usepackage[mathcal]{eucal}

\input xy
\xyoption{all}
\usepackage{epsfig}

\newtheorem{thm}{Theorem}[section]

\newtheorem{lem}[thm]{Lemma}

\newtheorem{question}[thm]{Question}

\theoremstyle{definition}
\newtheorem{defin}[thm]{Definition}

\theoremstyle{remark}
\newtheorem{remark}[thm]{Remark}

\numberwithin{equation}{section}

\newcommand{\delete}[1]{} 
\newcommand{\nt}{\noindent}

\def\eps{{\varepsilon}}
\newcommand{\sk}{\vskip 0.1cm}

\newcommand{\ben}{\begin{enumerate}}

\newcommand{\een}{\end{enumerate}}
\newcommand{\bit}{\begin{itemize}}

\newcommand{\eit}{\end{itemize}}

\def\R {{\mathbb R}}
\def\N {{\mathbb N}}
\def\Z {{\mathbb Z}}

\newcommand{\RUC}{\mathrm{RUC}}

\newcommand{\WAP}{\mathrm{WAP}}

\newcommand{\Asp}{\mathrm{Asp}}
\newcommand{\Tame}{\mathrm{Tame}}

\def\A{{\mathcal{A}}}

\def\F{{\mathcal F}}

\def\QED{\nobreak\quad\ifmmode\roman{Q.E.D.}\else{\rm Q.E.D.}\fi}

\def\a{\alpha}

\def\g{\gamma}

\newcommand{\cls}{\rm{cl\,}}

\newcommand{\lan}{\langle}
\newcommand{\ran}{\rangle}

\begin{document}

\title[]
{Tame functionals on Banach algebras}

\sk


\author[]{Michael Megrelishvili}
\address{Department of Mathematics,
Bar-Ilan University, 52900 Ramat-Gan, Israel}
\email{megereli@math.biu.ac.il}
\urladdr{http://www.math.biu.ac.il/$^\sim$megereli}

\date{October 2, 2017}

\begin{abstract}
In the present note 
we introduce \textit{tame functionals} on Banach algebras. A functional $f \in \A^*$ on a Banach algebra $\A$ is tame if the naturally defined linear operator $\A \to \A^*, a \mapsto f \cdot a$ factors through Rosenthal Banach spaces (i.e., not containing a copy of $l_1$). Replacing Rosenthal by reflexive we get a well known concept of weakly almost periodic functionals. 
So, always $\WAP(\A) \subseteq \Tame(\A)$. We show that tame functionals on $l_1(G)$ are induced exactly by \textit{tame functions} (in the sense of topological dynamics) on $G$ for every discrete group $G$. 
That is, $\Tame(l_1(G))=\Tame(G)$. Many interesting tame functions on groups come from dynamical systems theory. Recall that $\WAP(L_1(G))=\WAP(G)$ (Lau \cite{Lau77}, \"{U}lger \cite{Ulger86}) for every locally compact group $G$.  
It is an open question if  $\Tame(L_1(G))=\Tame(G)$ holds for (nondiscrete)  locally compact groups. 	 
\end{abstract}



\thanks{This research was partially supported by a grant of ISF 668/13 (Israel Science Foundation)}
\maketitle
\setcounter{tocdepth}{1}
\tableofcontents

\section{Introduction}

Weakly almost periodic  (WAP) functionals play a major role in the theory of Banach algebras. See for example Young
 \cite{Young}, Lau \cite{Lau77}, Filali, Neufang and Monfared \cite{FNM} and references therein. 
Recall that  a functional $\lambda \in \A^*$ on a Banach algebra $\A$ is said to be WAP (and write $\lambda \in \WAP (\A)$) 
	 if the natural linear operator  $$\mathfrak{L}_{\lambda}: \A \to \A^*, \ \ \ 
	a \mapsto  \lambda \cdot a$$
	factors through a reflexive Banach space; 
	notation: $\lambda \in \WAP(A)$. 
	As it follows from classical results \cite{DFJP}, 
	a characterization property is that the subset $\lambda \cdot B(\A)$ is weakly precompact in $\A^*$, where $B(\A)$ is the unit ball of $\A$ and for every $a \in \A$ the functional $\lambda \cdot a$ is defined as $\A \to \R, \ x \mapsto \lambda(a \star x)$, where $\star$ is the multiplication in $\A$.  	

Our aim is to introduce \textit{tame functionals}. In this case Rosenthal Banach spaces play the role of reflexive spaces. The technical characterization property comes combining some results of Rosenthal \cite{Ros0}, Talagrand \cite{Tal}, Saab \& Saab \cite{SS}. The criterion, Lemma \ref{l:RosBanSpCharact}, asserts that the subset $\lambda \cdot B(\A)$ is small in other sense, namely it is \textit{tame}, meaning that it does not contain an independent sequence of vectors, Definition \ref{d:tameF}. 
These concepts come from a celebrated $l_1$-dichotomy theorem of Rosenthal and play a major  role in several research lines in Banach space theory and also in topological dynamics. 

In the present note we 
study the tameness concept also in the theory of Banach algebras. 
The main result here is 
Theorem \ref{t:mainl_1} which asserts that tame functionals on the Banach algebra $l_1(G)$ are induced exactly by \textit{tame functions} (in the sense of topological dynamics) on $G$ for every discrete group $G$. 
That is, $\Tame(l_1(G))=\Tame(G)$. 
It should be compared with the well known formula for the WAP case. 
Namely, $\WAP(L_1(G))=\WAP(G)$ for every locally compact group $G$ (see Lau \cite{Lau77} and \"{U}lger \cite{Ulger86}). It is an open question if  $\Tame(L_1(G))=\Tame(G)$ holds for (nondiscrete)  locally compact $G$. 

Tame functions on topological groups (on the integers $\Z$ for instance) is the theme of quite intensive investigation in topological dynamics. Many interesting dynamical systems provide tame (not WAP) functions on acting groups. For example, this happens for Sturmian-like symbolic dynamical $G$-systems, with $G=\Z^k$ or (noncommutative modular group)  $G=PSL_2(\Z)$, \cite{GM-tame, GM-c}. So, by Theorem \ref{t:mainl_1} the class of tame functionals on $l_1(G)$ (which are not WAP functionals) is very rich. 

\section{Some definitions: independence and tameness} 

Let $l_{\infty}(X)$ be the Banach space (in the $\sup$-norm) of all bounded real valued functions on $X$. If $X$ is a topological space denote by $C(X)$ the Banach subspace of all continuous bounded functions. 

Let $\{f_n: X \to \R\}_{n \in \N}$ be a uniformly bounded sequence of real valued functions on a \emph{set} $X$. Following Rosenthal \cite{Ros0} we say that
this sequence is an \emph{$l_1$-sequence} 
if there exists a real constant $a >0$
such that for all $n \in \N$ and choices of real scalars $c_1, \dots, c_n$ we have
$$
a \cdot \sum_{i=1}^n |c_i| \leq ||\sum_{i=1}^n c_i f_i||_{\infty}.
$$

Then the closed linear span in $l_{\infty}(X)$
of the sequence $f_n$ be linearly homeomorphic to the Banach space $l_1$.
In fact, in this case the map $$l_1 \to l_{\infty}(X), \ \ (c_n)_{n \in \N} \to \sum_{n \in \N} c_nf_n$$ is a linear homeomorphic embedding.

Similarly can be defined a sequence of vectors in a Banach space which is equivalent to an  $l_1$-sequence.  
A Banach space $V$ is said to be {\em Rosenthal} if it does
not contain an isomorphic copy of $l_1$, or equivalently, if $V$ does not contain a sequence 
of vectors which is equivalent to an $l_1$-sequence. 

A Banach space $V$ is an {\em Asplund\/} space if the
dual of every separable Banach subspace is separable. Every Asplund space is Rosenthal 
and every reflexive space is Asplund. 

\subsection{Independence and tame systems} 
A sequence  $\{f_n: X \to \R\}_{n \in \N}$ of real valued functions on a set
$X$ is said to be \emph{independent} (see \cite{Ros0,Tal}) if
there exist real numbers $a < b$ such that
$$
\bigcap_{n \in P} f_n^{-1}(-\infty,a) \cap  \bigcap_{n \in M} f_n^{-1}(b,\infty) \neq \emptyset
$$
for all finite disjoint subsets $P, M$ of $\N$. 
Every bounded independent sequence is an $l_1$-sequence \cite{Ros0}.

	We say that a bounded family $F$ of real valued (not necessarily, continuous)  functions on a set $X$ is {\it tame}, \cite{GM-tame} if $F$ does not contain an independent sequence. 
For example, by \cite{Me-Helly} 
every bounded family of (not necessarily continuous) functions $[0,1] \to \R$ with total bounded variation (e.g., \textit{Haar systems}) is tame.  In fact this remains true replacing the set  $[0,1]$ by any circularly (e.g., linearly) ordered set.  

As to the negative examples. 
The sequence of projections on the Cantor cube $$\{\pi_n: \{0,1\}^{\N} \to \{0,1\}\}_{n \in \N}$$ 
and the sequence of Rademacher functions 
$$\{r_n: [0,1] \to \R\}_{n \in \N}, \ \ r_n(x):=sgn (\sin (2^n \pi x))$$
both are independent (hence, nontame).

%

\begin{defin}  \label{d:tameF} 
Let $V$ be a real Banach space and $M \subset V^*$ be a subset in the dual Banach space $V^*$. A  bounded  subset $F$ of $V$ is said to be \emph{tame for} $M$ if $F$, as a family of functions on $M$ is a tame family. If $F$ is tame for the unit ball $B(V^*)$ of $V^*$ (equivalently, for every bounded subset) then we simply say that $F$ is a \textit{tame subset in} $V$. 
\end{defin}

The family of tame subsets in a given Banach space is closed under taking subsets, finite unions, weak and norm closures and convex hulls. 
The case of subsets is trivial. The case of finite unions is easy (directly or using Lemma \ref{f:sub-fr}). For convex hulls use Lemma \ref{c:conv}.  
We explain here the case of the weak closure 
(which implies the norm closure case). 

 For definition and properties of \textit{fragmentability} see the Appendix below. 
 In particular, $\F(X)$ denotes the set of all fragmented maps $X \to \R$. If $X$ is compact this is exactly the set of functions with the point of continuity property (PCP), Lemma \ref{r:fr1}.1. 
	
\begin{lem} \label {l:w-clos}
	The weak closure of a tame subset $F$ of a Banach space $V$ is tame.  
\end{lem} 
\begin{proof}
	 Denote by $M:=\overline{F}^{w}$ the weak closure of $F$ in the Banach space $V$ and by $\overline{F}^p$ the pointwise closure of $F$ in $\R^{B(V^*)}$. Clearly,  $M \subset  C(B(V^*))$ and $M$ remains norm bounded. By Lemma \ref{f:sub-fr} it is enough to show that the pointwise closure $\overline{M}^{p}$ of $M$ in $\R^{B(V^*)}$ consist of fragmented maps. That is, $\overline{M}^p \subset \F(B(V^*))$. 
	 
	 Observe that 
	 $M \subset \overline{F}^p$. Indeed, since $F$ consist of linear maps $V^* \to \R$ and every linear  (not necessarily continuous) map on $V^*$ is uniquely defined by its restriction on the ball $B(V^*)$ it follows that 
	 the pointwise closure $\overline{F}^p$ of $F \subset \R^{B(V^*)}$ is the same as the pointwise closure of $F$ in $\R^{V^*}.$ Now we get that  $\overline{M}^p \subset \overline{F}^{p}$. 
	Since $F$ is a tame system on $B(V^*)$, by Lemma  \ref{f:sub-fr} this means that
	$\overline{F}^p \subset \F(B(V^*))$. Hence also $\overline{M}^p \subset \F(B(V^*))$, as desired. 	
\end{proof}

The following characterization of Rosenthal Banach spaces is a reformulation of some known results (see in particular, \cite{SS} and Lemma  \ref{f:sub-fr} below).

\begin{lem} \label{l:RosBanSpCharact} Let $V$ be a Banach space. 
	The following conditions are equivalent:  
	\ben
	\item $V$ is a Rosenthal Banach space; 
	\item The unit ball $B(V)$ is a tame subset of $V$;
	\item 
	Each $x^{**} \in V^{**}$ is a fragmented map when restricted to the
	weak${}^*$ compact ball $B(V^*)$ of $V^*$. Equivalently, $B(V^{**}) \subset \F(B(V^*))$.
	\een
\end{lem}

So, in particular, a Banach space $V$ is Rosenthal iff 
every bounded subset $F \subset V$ is tame  (as a family of functions) on every bounded subset $Y  \subset V^*$ of the dual space $V^*$, iff $F$ is eventually fragmented (Definition \ref{def:fr}.3) on $Y$.



Note also the following characterizations of Asplund and reflexive spaces. 
	 A Banach space $V$ is Asplund (reflexive) iff every bounded subset $F \subset V$ is a fragmented family  of functions (has Grothendieck's {\em Double Limit Property} (DLP)) on every bounded subset $Y \subset V^*$. 
	


%

Note that in all three cases the converse statements are true; as it follows from results of \cite{GM-tame} 
every tame (fragmented, DLP) bounded family $F$ of continuous functions on $X$ can be represented 
on a Rosenthal (Asplund, reflexive) Banach space. 
Recall that a representation of $F$ on 
a Banach space $V$ consists of
a pair $(\nu,\a)$ of bounded maps $$\nu: F \to V, \ \alpha: X \to V^*$$ where $\a$ is weak-star continuous and 
$$
f(x)= \langle \nu(f), \a(x) \rangle
\ \ \ \forall \ f \in F, \ \ \forall \ x \in X.
$$
In other words, the following diagram commutes
$$\xymatrix{ F \ar@<-2ex>[d]_{\nu} \times X
	\ar@<2ex>[d]^{\a} \ar[r]  & \R \ar[d]^{id } \\
	V \times V^* \ar[r]  &  \R }
$$


\subsection{Tame functionals} 
\begin{defin} \label{d:RosOp} 
	Let $T: V_1 \to V_2$ be a continuous linear operator between Banach spaces.  We say that $T$ is a \textit{Rosenthal operator} 
	if it factors through a Rosenthal Banach space. That is, there exists a Rosenthal Banach space $V_3$ and continuous linear operators $\a: V_1 \to V_3, \ \beta: V_3 \to V_2$ such that $T=\beta \circ \a$. 
\end{defin}

If either $V_1$ or $V_2$ are Rosenthal Banach spaces then every linear operator $V_1 \to V_2$ is Rosenthal. 
Note that if $T(V_1)$ is closed in $V_2$ (e.g., if $T$ is onto) then 
$T$ is a Rosenthal operator if and only if $T(V_1)$ is a Rosenthal Banach space.
 
\sk
For every functional $\lambda \in \A^*$ on a Banach algebra $\A$, with the multiplication $\star$, and every element $a \in \A$ 
we use the standard notation $\lambda \cdot a$ for the functional  
\begin{equation} \label{eq:functional} 
\lambda \cdot a: \A \to \R, \ x \mapsto \lambda (a \star x). 
\end{equation}

\begin{defin} \label{d:TameFunctional} 
	Let $\A$ be a Banach algebra and $\lambda \in \A^*$ be a functional. We say that $\lambda$ is \textit{tame} and write $\lambda \in \Tame (\A)$ if the natural continuous linear operator $$\mathfrak{L}_{\lambda}: \A \to \A^*, \ \ \ 
	a \mapsto  \lambda \cdot a$$
	is Rosenthal in the sense of Definition \ref{d:RosOp}; 
	notation: $\lambda \in \Tame(\A)$. 
\end{defin}
 
 \begin{remark}
 	Definition \ref{d:RosOp}, 
  without any name, appears in \cite{SS}. 
 Similarly can be defined \textit{Asplund and reflexive operators}. They lead (in the scheme of Definition \ref{d:TameFunctional}) to the corresponding definitions for functionals between Banach algebras. 
 In the latter case it is exactly \textit{weakly compact functional} as it follows from \cite{DFJP}. Clearly, 
 $$
 \WAP(\A) \subset \Asp(\A) \subset \Tame(\A). 
 $$
 \end{remark}

In general the inclusions are distinct (e.g., for $\A=l_1(\Z)$). 


The following technical lemma 
is a reformulation of  deep results by Rosenthal, Talagrand and Saab\&Saab.

\begin{lem} \label{l:TameFunctional} Let $T: V_1 \to V_2$ be a continuous linear operator between Banach spaces. 
The following conditions are equivalent: 
	\ben 
	\item $T: V_1 \to V_2$ is a Rosenthal operator; 
	\item For every $w^*$-compact 
	subset $M$ in $V_2^*$ and every $a^{**} \in V_1^{**}$ the restriction of $a^{**}$ to  $T^*(M)$ has a point of continuity; 
	\item $B(V_1)$ is a tame 
	family for $T^*(B(V_2^{*}))$ (Definition \ref{d:tameF}); 
	\item $T(B(V_1))$ is a tame family in $V_2$ (for $B(V_2^{*})$).
	\een
\end{lem}
\begin{proof}
	(1) $\Leftrightarrow$ (2) is a minor modification of \cite[Theorem 12]{SS}.
	Note that \cite[Theorem 12]{SS} was formulated for \emph{convex} subsets $M$. However, as the original proof from \cite{SS} demonstrates, we can write in the formulation \emph{arbitrary $w^*$-compact subset} $M$. 
	
	(2) $\Leftrightarrow$ (3) 
	First of all observe that it is enough to consider only a particular case of $M:=B(V_1^*)$. 
	By Goldstein's theorem the pointwise closure of $B(V_1)$, as a subset of functions on the weak-star compact space $B(V_1^*)$ (that is, the closure in the product space $\R^{B(V_1^*)}$), is the unit ball $B(V_1^{**})$ of the second dual.
	In fact, the same is true replacing $B(V_1^*)$ by $cB(V_1^*)$ for any given positive $c>0$. 
	
	 Every element in the pointwise closure of 
	$B(V_1)$, as a subset of functions on the weak-star compact space $T^*(B(V_2^{*}))$, can be viewed as a restriction of some $a^{**} \in V_1^{**}$. Indeed, since $T^*$ is bounded there exists a positive constant $c >0$ such that $T^*(B(V_2^{*})) \subset c B(V^*_1)$. Apply Lemma \ref{l:genQuot}.3 to the inclusion map 
	$$q: X_1=T^*(B(V_2^{*})) \hookrightarrow X_2=c B(V^*_1)$$
	with $X_1:=T^*(B(V_2^{*})) \subset X_2:=c B(V^*_1)$, $F_1:=B(V_1), \ F_2:=\{v \circ q: v \in B(V_1)\}$, where $v \circ q$ is a restriction of the map $v: cB(V_1^*) \to \R, \psi \mapsto \lan v,\psi \ran$ to $T^*(B(V_2^{*}))$. 
	
	 Now the equivalence (2) $\Leftrightarrow$ (4) of Lemma \ref{f:sub-fr} finishes the proof. 
	
	(3) $\Leftrightarrow$ (4) Lemma \ref{l:genQuot}.4(a) (for the map $q=T^*:  B(V_2^{*}) \to T^*(B(V_2^{*}))$) implies that 
	$B(V_1)$ is a Rosenthal family on $T^*(B(V_2^{*}))$ iff 
	$T(B(V_1))$ is a Rosenthal family on $B(V_2^{*})$. Now recall that Rosenthal family and tame family are the same in our setting by Lemma \ref{f:sub-fr}.  
\end{proof}

%

\sk

\section{Tame functionals on the group algebra $l_1(G)$} 

Let $G$ be a discrete group. Denote by $l_1(G)$ the usual (unital) Banach group algebra with respect to the convolution product $\star$ and the $l_1$-norm. 

For every $s \in G$ we have $\delta_s \in l_1(G)$, where $\delta_s(s)=1$ and $\delta_s(x)=0$ for every $x \in G$ with $x \neq s$. We have the embedding 
$$\delta: G \hookrightarrow l_1(G), \ \ s \mapsto \delta_s,$$ where 
\begin{equation}
\delta_s \star \delta_t=\delta_{st} \ \ \ \forall s,t \in G. 
\end{equation}  
For a group $G$ and a real-valued function $f: G \to \R$ the left translation $fs$ by $s \in G$ is defined as the function $$fs: G \to \R, \ (fs)(t)=f(st)$$ and $$fG:=\{fs: s \in G\}.$$ 

\nt Below the functional $f \cdot \delta_s$ for $f \in l_{\infty}(G)=l_1(G)^*$ is defined as in Equation \ref{eq:functional}. Then 
\begin{equation} 
(f \cdot \delta_s)(t)= (f \cdot \delta_s)(\delta_t) = f( \delta_s \star \delta_t)= f(\delta_{st})=f(st). 
\end{equation}
This means that 
\begin{equation} \label{l:simple} 
f \cdot \delta_s=fs, \ \ \ f \cdot \delta(G)=fG.
\end{equation}

In contrast to the algebra $L_1(G)$ for locally compact 
groups $G$, the algebra $l_1(G)$
has the unit element $\delta_e$, where $e$ is the unit element of $G$. 
Lemma \ref{l:NormDenseInl_1} implies that $span\{\delta(G)\}$ is norm dense in $l_1(G)$. 

Let 
$l_{\infty}(G)$ be the Banach space (usual $\sup$-norm) of all bounded functions on $G$. 
 For every continuous functional $f: l_1(G) \to \R$ we have the corresponding restriction $r(f):=f \circ \delta: G \to \R$. This defines the canonical isomorphism of Banach spaces $r: l_1(G)^* \to l_{\infty}(G)$. It reflects the standard fact that 
$l_1(G)^*$ can be identified with $l_{\infty}(G)$. 

For every $s \in G$ define the functional 
$$\sigma_s: l_{\infty} \to \R, \  h  \mapsto h(s).$$ 
We have an injection $$\sigma: G \to l_{\infty}^*(G), \  s \mapsto \sigma_s$$ such that $\sigma(G)$ is also discrete in the weak-star topology of $l_{\infty}^*$. 
Furthermore, $C(G)=l_{\infty}(G)$ as a Banach space is naturally isometric to the Banach space $C(\beta G)$, where, $\beta G$ is the Chech-Stone compactification of the discrete space $G$. Consider the natural topological embedding 
$i: \beta(G) \hookrightarrow l_{\infty}^*(G)=C(\beta G)^*$. Its restriction to $G$ is just $\sigma$. So, the weak-star closure of $\sigma(G)$ in $l_{\infty}^*(G)$ can be naturally identified with $\beta(G)$.

\subsection{Tame functions on groups}  
The theory of tame dynamical systems developed in a series of works 
(see e.g. \cite{Ko, Gl-tame,GM1,GM-rose,KL,GMU}).  
Connections to other areas of mathematics like: Banach spaces, 
circularly ordered systems, 
substitutions and tilings, quasicrystals, cut and project schemes and even model theory and logic were established. See e.g. \cite{Auj,Ibar, GM-c,GM-MTame} and the survey \cite{GM-survey} for more details. 


\begin{defin} \label{d:tame(G)} (see for example \cite{GM-survey,GM-MTame}) 
	Let $f: G \to \R$ be a bounded RUC (right uniformly continuous) function on a topological group $G$. Then $f$ is said to be a \textit{tame function} if $fG:=\{fs: s \in G\}$ is a tame family of functions on $G$;  
	notation: $f \in\Tame(G)$. 
\end{defin}

Tame functions on $G$ are characterized in \cite{GM-rose} as generalized matrix coefficients of isometric continuous representations of $G$ 
on Rosenthal Banach spaces. Recall a similar result from \cite{Me-nz} which asserts that weakly almost periodic functions on $G$ are exactly matrix coefficients of representations on reflexive spaces. So, it is immediate from these results that $\WAP(G) \subset \Tame(G)$. Note also that $\RUC(G)$ is exactly the set of all matrix coefficients for Banach representations of $G$. For more information about matrix coefficients of group representations we refer to \cite{Me-nz} and \cite{GalCourse}. 

There are many interesting tame functions on groups which are not WAP. For example, on the discrete integer group $\Z$. Fibonacci bisequence $c: \Z \to \{0,1\}$ defined by Fibonacci substitution is tame but not WAP. 

It is well known that $\WAP(L_1(G))=\WAP(G)$ for every locally compact group $G$ (see Lau \cite{Lau77} and \cite{Ulger86}).  In particular, for discrete $G$ it implies that $\WAP(l_1(G))=\WAP(G)$. Our main result in this note, Theorem \ref{t:mainl_1} below, shows that a similar result holds for the tame case. 

\begin{thm} \label{t:mainl_1} 
	$\Tame (l_1(G))=\Tame (G)$ for every discrete group $G$.
\end{thm}
\begin{proof} 
	
	 $\Tame (l_1(G))\subseteq \Tame (G)$. 
	
	Let $f \in \Tame (l_1(G))$. Then by Lemma \ref{l:TameFunctional} we know that the set $\mathfrak{L}_f(B(l_1(G))) = f \cdot B(l_1(G))$ is a tame subset in $l_{\infty}(G)$. Since $\delta (G) \subset B(l_1(G))$ we get that $f \cdot \delta (G) \subset f \cdot B(l_1(G))$ is also a tame subset  in $l_{\infty}(G)$. On the other hand, 
	$f \cdot \delta (G)= fG$ by Equation \ref{l:simple}. Hence, $fG$ is a tame subset as a family of functions on $B(l_{\infty}^*(G))$. Therefore, $fG$ is a tame subset also on the subset $G=\sigma(G) \subset B(l_{\infty}(G))$. This means that $f \in \Tame (G)$.

\vskip 0.5cm	
	
\nt	$\Tame (l_1(G))\supseteq \Tame (G)$. 
	
	 Let $f \in \Tame(G)$. Then $fG$ is a tame family of functions on $G$. That is, 
	$$fG = \mathfrak{L}_{f}(\delta(G))  \subset  l_{\infty}(G)$$ is a tame family on $\sigma(G) \subset l_{\infty}^*(G)= (l_1(G))^{**}$. Then by Lemma \ref{l_1}.3, $fG$ is a tame family also on $\beta G$. 
	Moreover, by Lemma \ref{p:from X to B*} we obtain that $fG$ is a tame family on the 
	weak-star closed unit ball $B(l_{\infty}(G)^*)$ of 
	$l_{\infty}^*(G)=C(\beta(G))^*$. Then the same is true for the union $-fG \cup fG$. 
	Now the convex hull $co(-fG \cup fG)$ is also tame family of functions on the compact space $B(l_{\infty}(G)^*)$ by Lemma \ref{c:conv}.
	
\begin{lem} \label{l:NormDenseInl_1}  
	$W:=co(-\delta(G) \cup \delta(G))$ is norm dense in the unit ball of $l_1(G)$. 
\end{lem} 
	\begin{proof} 
		Let $v \in l_1(G)$ with $||v|| \leq 1$.  
		 Then, by definition, $v$ is a  function $v: G \to \R$ such that the support $Supp(v)$ is at most countable, $v=\sum_{s \in Supp(v)} v(s) \delta_s$. It is common to write it simply  $v=\sum c_s \delta_s$, where $c_s=v(s)$. 
		  Then $||v||= \sum |c_s| \leq 1$. 

For a given $\eps >0$ we need to find $w \in W$ such that $||v-w|| < \eps$. There exists a finite subset $J \subset Supp(v)$ such that $||v - \Sigma_{s \in J} c_s \delta_{s}|| < \eps$. Define $w:=\Sigma_{s \in J} c_s \delta_{s}$. Now observe that $w \in W$. Indeed, we may suppose that $J=J_+ \cup J_-$ is the disjoint union where $J_+$ correspond to positive coefficients and $J_-$ to negative coefficients. Then 
		 $$w= \sum_{s \in J} c_s \delta_{s} = \sum_{s \in J_+} c_s \delta_{s} + \sum_{s \in J_-} (-c_s) (-\delta_{s}) + \frac{\Delta}{2} \delta_e + \frac{\Delta}{2} (-\delta_e) \in W,$$
		 where $$\Delta:=1 - \sum_{s \in J_+} c_s  + \sum_{s \in J_-} (-c_s) = 1-\sum_{s \in J} |c_s| \geq 0.$$ 
		\end{proof}
	
 We claim that the weak closure $\overline{f \cdot W}^{w}$ of $f \cdot W$ in $l_{\infty}(G)$ contains $f \cdot B_{l_1(G)}$. 
	Since $\mathfrak{L}_f$ is a linear operator, taking into account $f \cdot \delta_s=fs$, we have 
	$$f \cdot W =f \cdot  co (-\delta(G) \cup \delta(G))=co (f \cdot (-\delta(G) \cup \delta(G))=co(-fG \cup fG).$$	Lemma \ref{l:NormDenseInl_1} means that 
	$\overline{W}^{||\cdot||}=B(l_1(G))$. 
	Since the operator $\mathfrak{L}_f: l_1(G) \to l_{\infty}(G)$ is norm continuous, 
	the norm closure $\overline{f \cdot W}^{||\cdot||}$ of $\mathfrak{L}_f(W)=f \cdot W$ contains $f \cdot \overline{W}^{||\cdot||}$. Summing up we get 
	$$\overline{f \cdot W}^{w}	= \overline{co(-fG \cup fG)}^{w} = \overline{co(-fG \cup fG)}^{||\cdot||}= 
	\overline{f \cdot W}^{||\cdot||}  \supseteq f \cdot \overline{W}^{||\cdot||} = f \cdot B_{l_1}(G).
	$$
	
\sk 
	
	Now we use Lemma \ref{l:w-clos} which implies that the weak closure $\overline{co(-fG \cup fG)}^{w}$,  
	hence also its subfamily $f \cdot B(l_1(G))$, are tame families on the weak-star compact ball $B(l_{\infty}^*(G))$. 
	Finally Lemma \ref{l:TameFunctional} finishes the proof. 
\end{proof}

\sk

\begin{remark}  
	One may prove similarly the following formulas:  
	\ben
	\item 
	$\WAP (l_1(G))=\WAP (G)$.
	\item 
	$\Asp (l_1(G))=\Asp (G)$.
	\een
	
The second formula is new. In the first case we may use Grothendieck's  Double Limit Property characterization of weakly precompact subsets. In the second case -- properties of fragmentable families. Namely, one may show that $T$ is Asplund operator iff $T(B(V_1))$ is a fragmented family on $B(V_2^*)$. The proof is similar to the proof of Lemma \ref{l:TameFunctional}. 
\end{remark}

\begin{remark}
The sets $\WAP(G), \Asp(G), \Tame(G)$ are distinct even for the discrete group $\Z$. Indeed, the Fibonacci bisequence $c: \Z \to \{0,1\}$ is a tame function on $\Z$ but not Asplund (see \cite{GM-c,GM-MTame}).  The characteristic function $\chi_{\N}: \Z \to \{0,1\}$ is Asplund but not WAP. Hence, the sets of functionals $\WAP(l_1(\Z)), \Asp(l_1(\Z)), \Tame(l_1(\Z))$ are also distinct. 
\end{remark}


We are going to investigate tame functionals in some future works. Among others we intend to deal with the following natural question. 

\begin{question}
	Is it true that $\Tame(L_1(G))=\Tame(G)$ for every (nondiscrete) locally compact group $G$? 
\end{question}


\sk

\section{Appendix} \label{s:app}

\subsection{Background on fragmentability and tame families}
\label{s:Ban}

The following definitions provide natural generalizations of the fragmentability concept \cite{JR}. 

\begin{defin} \label{def:fr}
	Let $(X,\tau)$ be a topological space and
	$(Y,\mu)$ a uniform space.
	\ben
	\item \cite{JOPV,me-fr} 
	$X$ is {\em $(\tau,
		\mu)$-fragmented\/} by a 
	(typically, not continuous)
	function $f: X \to Y$ if for every nonempty subset $A$ of $X$ and every $\eps
	\in \mu$ there exists an open subset $O$ of $X$ such that $O \cap
	A$ is nonempty and the set $f(O \cap A)$ is $\eps$-small in $Y$.
	We also say in that case that the function $f$ is {\em
		fragmented\/}. Notation: $f \in {\mathcal F}(X,Y)$, whenever the
	uniformity $\mu$ is understood.
	If $Y=\R$ then we write simply ${\mathcal F}(X)$.  
	
	\item \cite{GM1}
	We say that a {\it family of functions} $F=\{f: (X,\tau) \to
	(Y,\mu) \}$ is {\it fragmented} 
	if condition (1) holds simultaneously for all $f \in F$. That is, 
	$f(O \cap A)$ is $\eps$-small for every $f \in F$.
	\item \cite{GM-rose}
	We say that $F$ is an \emph{eventually fragmented family} 
	if every infinite subfamily $C \subset F$ contains
	an infinite fragmented subfamily $K \subset C$.
	\een
\end{defin}

In Definition \ref{def:fr}.1 when $Y=X, f={id}_X$ and $\mu$ is a
metric uniformity, we retrieve the usual definition of
fragmentability (more precisely, $(\tau,\mu)$-fragmentability) in the sense of Jayne and Rogers \cite{JR}.
Implicitly it already appears in a paper of Namioka and Phelps \cite{NP}. 

\begin{lem} \label{r:fr1} \cite{GM1, GM-rose}
	\ben
	
	\item If $f: (X,\tau) \to (Y,\mu)$ has 	
	\emph{a point of continuity property} PCP (i.e., for every closed
	nonempty $A \subset X$ the restriction $f_{|A}: A \to Y$ has a continuity point) 
	then it is fragmented. If $(X,\tau)$ is hereditarily Baire (e.g., compact, or Polish) and $(Y,\mu)$ is a pseudometrizable uniform space then $f$ is fragmented if and only if $f$ has PCP. 
	So, in particular, for compact $X$, the set $\F(X)$ is exactly $B'_r(X)$ in the notation of \cite{Tal}. 

	\item
	If $X$ is Polish and $Y$ is a separable metric space then
	$f: X \to Y$ is fragmented iff $f$ is a Baire class 1 function (i.e., the inverse image of every open set is
	$F_\sigma$).
	
	\een 
\end{lem}

For other properties of fragmented maps and fragmented families
we refer to \cite{N, JOPV, me-fr, Me-nz, GM1, GM-rose, GM-tame}.  
Basic properties and applications of fragmentability in topological dynamics can be found in \cite{GM-rose, GM-survey, GM-tame}.

\sk  
\subsection{Independent sequences of functions}
\label{s:ind}

%

The following useful theorem synthesizes some known results.
It mainly is based on results of Rosenthal and Talagrand. 
The equivalence of (1), (3) and (4) is a part of \cite[Theorem 14.1.7]{Tal}
For the case (1) $\Leftrightarrow$ (2) note that every bounded independent sequence $\{f_n: X \to \R\}_{n \in \N}$ is 
an $l_1$-sequence (in the $\sup$-norm), \cite[Prop. 4]{Ros0}. On the other hand, as the proof of \cite[Theorem 1]{Ros0} shows, if $\{f_n\}_{n \in \N}$ has no independent subsequence then it has a pointwise convergent subsequence. Bounded pointwise-Cauchy sequences in $C(X)$ (for compact $X$) are weak-Cauchy as it follows by Lebesgue's theorem. Now Rosenthal's dichotomy theorem \cite[Main Theorem]{Ros0} asserts that $\{f_n\}$ has no $l_1$-sequence.  
In \cite[Sect. 4]{GM-rose} we show why eventual fragmentability of $F$ can be included in this list (item (5)).  

\begin{lem} \label{f:sub-fr} 
	Let $X$ be a compact space and $F \subset C(X)$ a bounded subset.
	The following conditions are equivalent:
	\begin{enumerate}
		\item
		$F$ does not contain an $l_1$-sequence. 
		\item $F$ is a tame family (does not contain an independent sequence). 
		\item
		Each sequence in $F$ has a pointwise convergent subsequence in $\R^X$.
		\item $F$ is a \emph{Rosenthal family} (meaning that the pointwise closure ${\cls}(F)$ of $F$ in $\R^X$ consists of fragmented maps; that is,
		${\cls}(F) \subset {\mathcal F}(X)$). 
		\item $F$ is an eventually fragmented family.
	\end{enumerate}
\end{lem}

\begin{lem} \label{l_1} \
	\ben
	\item
	Let $q: X_1 \to X_2$ be a map between sets and
	$\{f_n: X_2 \to \R\}_{n \in \N}$ a bounded sequence of functions \emph{(with no continuity assumptions on $q$ and $f_n$)}.
	If $\{f_n \circ q\}$ is an independent sequence on $X_1$
	then $\{f_n\}$ is an independent sequence on $X_2$.
	\item
	If $q$ is onto then the converse is also true. That is $\{f_n \circ q\}$ is independent if and only if $\{f_n\}$ is independent.
	\item
	Let $\{f_n\}$ be a bounded sequence of continuous
	functions on a topological space $X$.
	Let $Y$ be a \emph{dense} subset of $X$. Then $\{f_n\}$ is an independent sequence on $X$ if and only if
	the sequence of restrictions
	$\{f_n|_Y\}$ is an independent sequence on $Y$.
	\een
\end{lem}
\begin{proof}
	Claims (1) and (2) are straightforward.
	
	(3)
	Since $\{f_n\}$ is an independent sequence 
	for every pair of finite disjoint sets
	$P, M \subset \N$,
	the set
	$$
	\bigcap_{n \in P} f_n^{-1}(-\infty,a) \cap  \bigcap_{n \in M} f_n^{-1}(b,\infty)
	$$
	is non-empty. 
	This set is open because every $f_n$ is continuous. 
	Hence, each of them meets the dense set $Y$.
	As
	$f_n^{-1}(-\infty,a) \cap Y= f_n|_Y^{-1}(-\infty,a)$ and
	$f_n^{-1}(b,\infty) \cap Y= f_n|_Y^{-1}(b,\infty)$,
	this implies that $\{f_n|_Y\}$ is an independent sequence on $Y$.
	
	Conversely if $\{f_n|_Y\}$ is an independent sequence on a subset $Y \subset X$ then by (1)
	(where $q$ is the embedding $Y \hookrightarrow X$),
	$\{f_n\}$ is an independent sequence on $X$.
\end{proof}

\begin{lem} \label{l:genQuot} 
	Let $q: X_1 \to X_2$ be a map between sets. 
	\ben
	\item
	The natural map
	$\gamma: \R^{X_2} \to \R^{X_1}, \  \gamma(\phi)=\phi \circ q$
	is pointwise continuous.
	\item If $q: X_1 \to X_2$ is onto then $\gamma$ is injective.
	\item 
	$F_2 \subset [a,b]^{X_2}$ and $F_1 \subset [a,b]^{X_2}$ be subsets such that $F_1=F_2 \circ q$. Then $\g(\cls_p(F_2))=\cls_p(F_1)$.
	\item
	Let $q: X_1 \to X_2$ be a continuous onto map between compact spaces, 
	$F_2 \subset C(X_2)$ and $F_1 \subset C(X_1)$ be norm
	bounded subsets such that $F_1=F_2 \circ q$. Then
	\begin{itemize}
		\item [(a)]
		$F_1$ is a Rosenthal family for $X_1$ if and only if $F_2$ is a
		Rosenthal family for $X_2$. 
		\item [(b)]
		$\gamma$ induces a
		homeomorphism between the compact spaces $\cls_p(F_2)$ and
		$\cls_p(F_1)$.
	\end{itemize}
	\een
\end{lem}
\begin{proof}  Claims (1) and (2) are trivial.
	
	(3)(a): By the continuity of $\gamma$ we get $\gamma(F_2) \subset
	\gamma(\cls_p(F_2)) \subset \cls_p(\gamma(F_2))$. Since $F_2$ is
	bounded the set $\cls_p(F_2)$ is compact in $\R^{X_2}$. Then
	$\gamma(\cls_p(F_2))=\cls_p(\gamma(F_2))$. On the other hand,
	$\gamma(F_2) = F_2 \circ q = F_1$. Therefore, $\cls_p(F_2) \circ q
	= \cls_p(F_1)$. 
	
	(4)(a) Use (3) and Lemma \ref{l:factor}.
	
	(3)(b): Combine the assertions (1) and (2) taking into account
	that $\gamma(\cls_p(F_2))=\cls_p(\gamma(F_2))=\cls_p(F_1)$.
\end{proof}


%

Recall some useful Lemmas from \cite{GM-rose}. 

\begin{lem}  \label{l:factor} 
	Let $\a: X \to X'$ be a continuous onto map between compact
	spaces. Assume that $(Y, \mu)$ is a uniform space, $f: X \to Y$
	and $f': X' \to Y$ are maps such that $f' \circ \a=f$. Then $f$
	is a fragmented map if and only if $f'$ is a fragmented map.
\end{lem}

\begin{lem} \label{c:conv} 
	Let $F \subset C(X)$ be a tame (eq. Rosenthal) family for a compact space $X$. Then its
	convex hull ${co}(F)$ is also a tame family for $X$.
\end{lem}

\begin{lem} \label{p:from X to B*} 
	Let $X$ be a compact space and $F \subset C(X)$ be a bounded family.
	Then  
$F$ is a tame family for $X$ if and only if $F$ is a tame family for the weak-star compact unit ball $B(C(X)^*)$.
\end{lem}



\bibliographystyle{amsplain}

\end{document}